\documentclass[11pt,reqno]{amsart}

\usepackage[T1]{fontenc}
\usepackage[utf8]{inputenc}
\usepackage{newtxtext}
\usepackage{amsmath,amssymb,amsfonts,mathtools}
\usepackage{newtxmath}
\usepackage[letterpaper,centering,textwidth=5.95in,textheight=8.65in]{geometry}
\usepackage{booktabs,array,tabularx,capt-of}
\usepackage{tikz}
\usepackage{float}
\usetikzlibrary{calc,positioning,decorations.pathreplacing}
\usepackage{enumitem}
\usepackage{microtype}
\usepackage{xcolor}
\usepackage{hyperref}
\hypersetup{hidelinks}

\hypersetup{
  pdftitle={A Proof of Sundaram's Bounded-Interval Higher Lie Positivity Conjecture},
  pdfauthor={},
  pdfsubject={Schur positivity for bounded partial sums of higher Lie characteristics},
  pdfkeywords={Schur positivity, symmetric functions, symmetric group characters, higher Lie modules, Foulkes characters, cyclic induction}
}

\newcommand{\F}{F}
\newcommand{\Lie}{\operatorname{Lie}}
\newcommand{\ch}{\operatorname{ch}}
\newcommand{\1}{\mathbf 1}
\newcommand{\eps}{\varepsilon}
\newcommand{\wpr}{\omega_{\mathrm{pr}}}

\newtheorem{theorem}{Theorem}[section]
\newtheorem{lemma}[theorem]{Lemma}
\newtheorem{proposition}[theorem]{Proposition}
\newtheorem{corollary}[theorem]{Corollary}
\theoremstyle{definition}
\newtheorem{definition}[theorem]{Definition}
\newtheorem{remark}[theorem]{Remark}

\title[BOUNDED-INTERVAL HIGHER LIE POSITIVITY]
{A PROOF OF SUNDARAM'S BOUNDED-INTERVAL HIGHER LIE POSITIVITY CONJECTURE}


\author{Kesen Hou}

\date{July 14th, 2026}

\begin{document}

\begin{abstract}
For positive integers $n$ and $M$, Sundaram defined
\[
 F_{n,M}=\sum_{\substack{d\mid n\\d\le M}}\Lie_{n/d}[p_d]
\]
and conjectured that every coefficient in its Schur expansion is a nonnegative integer. We prove the conjecture. The power-sum expansion isolates the contribution of the identity class, and the proof divides Young diagrams at the intrinsic majority threshold. If the first row or first column contains more than half of the boxes, the Bernstein creation formula separates each rectangular-cycle character value into a hook constant and a remainder supported only on short cycles. The restrictions of the trivial and sign characters to the cyclic subgroup evaluate the constants exactly, whereas a uniform binomial contraction controls every distance from the boundary at once. If neither a row nor a column has a majority, Swanson's opposite-hook estimate gives a Catalan-scale lower bound for the dimension. The Fomin--Lulov rectangular-character estimate, together with uniform bounds for the truncated M\"obius weights, then shows that the identity class dominates. Sundaram's plethystic identity consequently implies that $\prod_{r=1}^M(1-p_r)^{-1}$ is Schur-positive for every fixed $M$. We also record a degree-$12$ counterexample to a separate congruence-class product conjecture attributed to Richard Stanley.
\end{abstract}

\maketitle

\section{Introduction}

The character table of the symmetric group $S_n$ is naturally indexed by partitions in two different ways. A partition $\lambda\vdash n$ is a weakly decreasing sequence of positive integers with sum $n$; its Young diagram has $\lambda_i$ boxes in row $i$. The irreducible representations of $S_n$ are indexed by such partitions $\lambda$, while the conjugacy classes are indexed by partitions $\rho$ recording the cycle lengths of a permutation. We write $\chi^\lambda(\rho)$ for the value of the irreducible character indexed by $\lambda$ on the class of cycle type $\rho$.

A basic positivity question is the following. Given a natural collection $\mathcal C$ of cycle types, when is
\[
 \sum_{\rho\in\mathcal C}\chi^\lambda(\rho)
\]
a nonnegative integer for every $\lambda\vdash n$? The individual character values may have either sign, so this is generally a question about cancellation across the chosen classes.

Symmetric functions package all these character sums into a single object. In degree $n$, the power-sum functions $p_\rho$ are indexed by conjugacy classes and the Schur functions $s_\lambda$ are indexed by irreducible representations. Every homogeneous symmetric function $g$ of degree $n$ has a unique expansion
\[
 g=\sum_{\lambda\vdash n}b_\lambda s_\lambda,
\]
called its \emph{Schur expansion}. The function $g$ is \emph{Schur-positive} when all $b_\lambda$ are nonnegative integers. The Frobenius characteristic map gives the fundamental change-of-basis formula
\begin{equation}
 p_\rho=\sum_{\lambda\vdash n}\chi^\lambda(\rho)s_\lambda
 \qquad(\rho\vdash n).
 \label{eq:power-schur}
\end{equation}
Consequently,
\[
 \left[s_\lambda\right]\sum_{\rho\in\mathcal C}p_\rho
 =\sum_{\rho\in\mathcal C}\chi^\lambda(\rho).
\]
Moreover, a homogeneous symmetric function is Schur-positive precisely when it is the Frobenius characteristic of an actual finite-dimensional $S_n$-representation; the coefficient of $s_\lambda$ is then the multiplicity of the irreducible representation indexed by $\lambda$. Thus Schur positivity is the symmetric-function form of the assertion that the relevant character sums are the multiplicities of a genuine representation. We refer to \cite{Macdonald,Sagan,Stanley} for the standard background.

Fix $M\ge1$ and consider
\begin{equation}
 P_M:=\prod_{r=1}^{M}(1-p_r)^{-1}
     =\sum_{\rho:\,\rho_1\le M}p_\rho.
 \label{eq:PM}
\end{equation}
Indeed, expanding $(1-p_r)^{-1}=1+p_r+p_r^2+\cdots$ records how many cycles of length $r$ occur. Hence, the degree-$n$ part of $P_M$ is the sum of $p_\rho$ over all conjugacy classes of $S_n$ whose cycles have length at most $M$. By \eqref{eq:power-schur}, its $s_\lambda$-coefficient is the sum of $\chi^\lambda(\rho)$ over precisely those classes. There is no termwise positivity: the problem is whether the full bounded-cycle sum is nevertheless nonnegative for every irreducible character.

The second ingredient is the classical Lie representation. Let $\mathcal L_m$ be the multilinear component of the free Lie algebra on generators $x_1,\ldots,x_m$: it is spanned by Lie monomials in which every generator occurs exactly once. Permuting the generators gives an $S_m$-action, and its Frobenius characteristic is denoted by $\Lie_m$. Thrall's higher Lie modules arise by forming symmetrized tensor constructions from these multilinear Lie pieces. In the language of symmetric functions, such constructions are expressed by \emph{plethysm}. For the present introduction it is enough to regard $f[g]$ as a composition operation on symmetric functions that encodes the corresponding wreath-product induction; the precise conventions are recalled in Section~\ref{sec:symmetric}. The functions $\Lie_{n/d}[p_d]$ occur as the rectangular-cycle building blocks in characters induced from the cyclic subgroup generated by an $n$-cycle; see \cite{Thrall,Foulkes,Klyachko,Jollenbeck,Reutenauer,Sundaram2020,Sundaram2021}.

For a general set $T$ of allowed cycle lengths, Sundaram introduced homogeneous functions $f_n^T$ whose symmetric powers recover the product $\prod_{m\in T}(1-p_m)^{-1}$. For the bounded interval $T=[M]=\{1,\ldots,M\}$, her construction gives
\begin{equation}
 \F_{n,M}:=\sum_{\substack{d\mid n\\d\le M}}\Lie_{n/d}[p_d]
 \label{eq:F-definition}
\end{equation}
and the completed identity
\begin{equation}
 H\!\left[\sum_{n\ge1}\F_{n,M}\right]=P_M,
 \qquad H=1+h_1+h_2+\cdots.
 \label{eq:Sundaram-identity}
\end{equation}
When $G$ is the Frobenius characteristic of a representation, $h_j[G]$ is the characteristic of the corresponding symmetric-power/wreath-product construction. Therefore, if every homogeneous input $\F_{n,M}$ is Schur-positive, then the right-hand side of \eqref{eq:Sundaram-identity} is Schur-positive for every degree; this completed-series implication is stated explicitly in \cite[Corollary~3.1]{Sundaram2019b}.

The two extreme cutoffs already have familiar representation-theoretic meanings. For $M=1$, one has $\F_{n,1}=\Lie_n$. For $M\ge n$, every divisor of $n$ occurs, and $\F_{n,M}$ is the Frobenius characteristic of the representation induced from the trivial character of the cyclic subgroup generated by an $n$-cycle. An intermediate cutoff is more subtle: the selected divisors need not be the complete divisor set of any single divisor of $n$. For example, when $(n,M)=(12,3)$ the selected set is $\{1,2,3\}$, whereas no divisor of $12$ has exactly these divisors. Thus, the standard cyclic-induction formulas do not by themselves produce a representation for the truncated sum.

Sundaram's bounded-interval problem belongs to a broader program on Schur-positive sums of power sums, the conjugacy action of $S_n$, and variants of the Lie representation. The representation-theoretic framework for sums over conjugacy classes was developed in \cite{Sundaram2018}; interval sums in the character table were studied in \cite{Sundaram2019a}; the general functions $f_n^T$ and the plethystic product formula were introduced in \cite{Sundaram2019b}; and the associated Lie variants and prime-power higher Lie modules were developed in \cite{Sundaram2020,Sundaram2021}. The bounded-interval assertion was posed in \cite[Conjecture 5]{Sundaram2019b} and appears in its later formulation as \cite[Conjecture 4]{Sundaram2021}.

\begin{theorem}[Main theorem]
For all integers $n,M\ge1$, the symmetric function $\F_{n,M}$ is Schur-positive. Equivalently,
\[
 [s_\lambda]\F_{n,M}\in\mathbb Z_{\ge0}
 \qquad(\lambda\vdash n).
\]
\label{thm:main}
\end{theorem}

We next describe the mechanism of the proof. The power-sum formula established in Section~\ref{sec:symmetric} yields
\begin{equation}
 n[s_\lambda]\F_{n,M}
 =f^\lambda+\sum_{\substack{d\mid n\\d>1}}
   \psi_M(d)\chi^\lambda(d^{n/d}),
 \qquad
 \psi_M(d)=\sum_{\substack{m\mid d\\m\le M}}m\mu(d/m),
 \label{eq:intro-central}
\end{equation}
where $f^\lambda=\chi^\lambda(1^n)$ is the dimension of the irreducible representation indexed by $\lambda$, and $d^{n/d}$ denotes the cycle type consisting of $n/d$ cycles of length $d$. The first term in \eqref{eq:intro-central} comes from the identity permutation and is positive. Our aim is to show that the weighted contribution of all nonidentity rectangular cycle types cannot outweigh it. A decomposition into Foulkes characters separately proves that the coefficients are integers, so the only issue that remains is to establish their signs.

The proof uses the intrinsic division of Young diagrams into two regions. Write $\lambda'$ for the conjugate partition, so that $\lambda_1$ and $\lambda_1'$ are the lengths of the first row and first column.

If $\max(\lambda_1,\lambda_1')>n/2$, the diagram has a \emph{majority row or column}. Suppose, for example, that $\lambda=(n-r,\mu)$ with $r<n/2$. The Bernstein creation formula separates $\chi^\lambda(d^{n/d})$ into a constant term, which occurs only for hooks, and a remainder that vanishes whenever $d>r$. The constant term is evaluated exactly by restricting the trivial or sign character to the long cyclic subgroup. The remaining terms therefore involve only short cycles, and one uniform binomial contraction controls them for every value of $r$; no fixed-width boundary analysis is needed.

If $\lambda_1,\lambda_1'\le n/2$, the diagram is \emph{balanced}. Swanson's opposite-hook estimate gives the Catalan-scale dimension bound
\[
 f^\lambda\ge
 \frac1{m_n+1}\binom n{m_n},
 \qquad m_n=\left\lfloor\frac{n-1}{2}\right\rfloor.
\]
The Fomin--Lulov estimate then bounds the character values on every nonidentity rectangular cycle type by a fractional power of $f^\lambda$. Uniform estimates for the truncated M\"obius weights $\psi_M(d)$ show that the positive identity contribution in \eqref{eq:intro-central} strictly dominates the entire remainder. This argument is global: it does not require a list of low-degree partitions or a table of individual character values.

Sections~\ref{sec:symmetric} and~\ref{sec:foulkes} establish the symmetric-function identities, integrality, and the divisor-prefix reductions. Section~\ref{sec:estimates} records the arithmetic and rectangular-character estimates. Sections~\ref{sec:majority} and~\ref{sec:balanced} treat majority and balanced diagrams, respectively. Section~\ref{sec:completion} completes the proof and applies Sundaram's identity. The final section concerns a separate congruence-class conjecture.

\section{Symmetric functions, characters, and completion}
\label{sec:symmetric}

Write $\Lambda_{\mathbb Q}=\bigoplus_{n\ge0}\Lambda^n_{\mathbb Q}$ for the graded ring of symmetric functions. We use the usual partition notation $\lambda\vdash n$, write $\lambda'$ for conjugation, and put $h_j=s_{(j)}$, $e_j=s_{(1^j)}$. The Hall inner product is normalized by
\[
 \langle s_\lambda,s_\mu\rangle=\delta_{\lambda\mu},
 \qquad
 \langle p_\rho,s_\lambda\rangle=\chi^\lambda(\rho).
\]
Thus $[s_\lambda]g=\langle g,s_\lambda\rangle$. We shall use the following representation-theoretic argument repeatedly of this coefficient. If $W$ is an $S_n$-module with character $\chi^W$, then
\begin{equation}
 [s_\lambda]\ch W
 =\langle\chi^W,\chi^\lambda\rangle_{S_n}
 =\dim\operatorname{Hom}_{S_n}(S^\lambda,W).
 \label{eq:Schur-multiplicity}
\end{equation}
In particular, if $W=\operatorname{Ind}_{H}^{S_n}U$, then the Frobenius reciprocity gives
\begin{equation}
 [s_\lambda]\ch W
 =\left\langle\chi^U,
   \operatorname{Res}_{H}^{S_n}\chi^\lambda\right\rangle_H.
 \label{eq:Frobenius-coefficient}
\end{equation}
Thus, a Schur coefficient of an induced characteristic is computed by an explicit inner product of the character in $H$; when $U$ is irreducible, it is the multiplicity of $U$ within the restriction of $S^\lambda$.

We work in the degree completion
\[
 \widehat\Lambda_{\mathbb Q}:=\prod_{n\ge0}\Lambda^n_{\mathbb Q}.
\]
Its ring operations are defined degree by degree.

\begin{definition}
An element $F=\sum_{n\ge0}F_n\in\widehat\Lambda_{\mathbb Q}$ is Schur-positive if every homogeneous component $F_n$ is Schur-positive.
\end{definition}

We use plethystic notation throughout; in particular, $g[p_d]$ is obtained in the power-sum basis by replacing $p_j$ with $p_{jd}$. If $G$ has no constant term, then $H[G]=\sum_{j\ge0}h_j[G]$ is well defined in $\widehat\Lambda_{\mathbb Q}$: since every term of $G$ has the positive degree, only finitely many values of $j$ can contribute to any fixed homogeneous degree.

We shall use one standard plethysm-positivity fact. If $f$ and $g$ are finite nonnegative integral combinations of Schur functions, then $f[g]$ is again such a combination; see \cite[Chapter~7, Appendix~2, Theorem~A2.7]{Stanley}, or the equivalent wreath-product realization in \cite[Chapter~I, \S8 and Appendix~A, \S6]{Macdonald}.

Applied degree by degree, this gives the same statement in the completion: if $G=\sum_{n\ge1}G_n$ has Schur-positive homogeneous components, then $H[G]$ is Schur-positive. Sundaram discussed this completed-series implication in the present setting in \cite[Corollary~3.1]{Sundaram2019b}. This is the only plethysm-positivity input used below.

For a set $T\subseteq\mathbb N_{\ge1}$ containing $1$, Sundaram defines
\begin{equation}
 \psi_T(d)=\sum_{\substack{m\in T\\m\mid d}}m\mu(d/m),
 \qquad
 f_n^T=\frac1n\sum_{d\mid n}\psi_T(d)p_d^{n/d}.
 \label{eq:Sundaram-fT}
\end{equation}
Here $\mu$ is the number-theoretic M\"obius function. Sundaram proves
\begin{equation}
 H\!\left[\sum_{n\ge1}f_n^T\right]
 =\prod_{m\in T}(1-p_m)^{-1};
 \label{eq:Sundaram-general}
\end{equation}
see \cite[Theorem 3.1]{Sundaram2019b} and \cite[Theorem 4.4 and Corollary 4.5]{Sundaram2021}.

For $T=[M]=\{1,\ldots,M\}$, the classical Lie characteristic
\[
 \Lie_m=\frac1m\sum_{e\mid m}\mu(e)p_e^{m/e}
\]
gives
\[
 \Lie_{n/m}[p_m]
 =\frac{m}{n}\sum_{e\mid n/m}\mu(e)p_{me}^{n/(me)}
 \qquad(m\mid n).
\]
Fix $c\mid n$.  The monomial $p_c^{n/c}$ occurs in this expression precisely when
$m\mid c$ and $e=c/m$, and its coefficient is then $m\mu(c/m)/n$.  Summing over the
selected divisors $m\le M$ therefore gives
\begin{equation}
 \F_{n,M}=f_n^{[M]}
 =\frac1n\sum_{c\mid n}\psi_M(c)p_c^{n/c},
 \qquad
 \psi_M(c)=\sum_{\substack{m\mid c\\m\le M}}m\mu(c/m).
 \label{eq:power-formula}
\end{equation}
Pairing with $s_\lambda$ gives the coefficient formula that drives the proof. The $d=1$ term is the identity class: since $\psi_M(1)=1$, it contributes
\[
 \langle p_1^n,s_\lambda\rangle
 =\chi^\lambda(1^n)=f^\lambda,
\]
the dimension of $S^\lambda$. Every term with $d>1$ comes from the rectangular cycle type $d^{n/d}$. Hence
\begin{equation}
 n[s_\lambda]\F_{n,M}
 =f^\lambda+\sum_{\substack{d\mid n\\d>1}}
   \psi_M(d)\chi^\lambda(d^{n/d}).
 \label{eq:central-character}
\end{equation}

\section{Foulkes characters, integrality, and divisor prefixes}
\label{sec:foulkes}

Let $\sigma_n=(1\ 2\ \cdots\ n)$ and $C_n=\langle\sigma_n\rangle$. After choosing a primitive $n$th root of unity $\zeta_n$, let $\chi_n^{(k)}$ be the one-dimensional character of $C_n$ given by $\sigma_n\mapsto\zeta_n^k$, and put
\[
 \ell_n^{(k)}:=\ch\operatorname{Ind}_{C_n}^{S_n}\chi_n^{(k)}.
\]
Foulkes's decomposition is
\begin{equation}
 \ell_n^{(k)}=\sum_{d\mid(k,n)}\Lie_{n/d}[p_d].
 \label{eq:Foulkes}
\end{equation}
See \cite{Foulkes,Klyachko} and \cite[Corollary~4.12]{Sundaram2021}. Formula~\eqref{eq:Foulkes} is the bridge between divisor structure and genuine representations: if you the selected Lie pieces form the complete divisor set of one parameter $k$, their sum is automatically a character. For a general cutoff, the next M\"obius inversion rewrites the truncation as an integral combination of these Foulkes characters.

Assume $M<n$. For $k\mid n$, define
\[
 \theta_{M,n}(k):=\sum_{\substack{j\mid n/k\\kj\le M}}\mu(j).
\]

\begin{lemma}[Foulkes decomposition of the cutoff]
For $M<n$,
\begin{equation}
 \F_{n,M}=\sum_{k\mid n}\theta_{M,n}(k)\ell_n^{(k)},
 \qquad
 \sum_{k\mid n}\theta_{M,n}(k)=1.
 \label{eq:theta-decomposition}
\end{equation}
\label{lem:theta}
\end{lemma}

\begin{proof}
Fix $d\mid n$. On the right-hand side of \eqref{eq:theta-decomposition}, the coefficient of $\Lie_{n/d}[p_d]$ is
\[
 \sum_{d\mid k\mid n}\theta_{M,n}(k)
 =\sum_{t\mid n/d}\ \sum_{\substack{j\mid n/(dt)\\dtj\le M}}\mu(j).
\]
Setting $a=tj$ transforms this into
\[
 \sum_{\substack{a\mid n/d\\da\le M}}\ \sum_{j\mid a}\mu(j).
\]
The defining inversion property of the M\"obius function is
\[
 \sum_{j\mid a}\mu(j)=
 \begin{cases}
 1,&a=1,\\
 0,&a>1.
 \end{cases}
\]
Thus only $a=1$ can survive, and it is included exactly when $d\le M$. This proves the first identity. The second is the same coefficient comparison at $d=1$.
\end{proof}

The first consequence is integrality. Each $\theta_{M,n}(k)$ is an integer, and every $\ell_n^{(k)}$ is the Frobenius characteristic of the actual induced representation $\operatorname{Ind}_{C_n}^{S_n}\chi_n^{(k)}$. Equation~\eqref{eq:theta-decomposition} therefore expresses each Schur coefficient of $\F_{n,M}$ as an integer linear combination of integer multiplicities. So all Schur coefficients are integers, even though the weights $\theta_{M,n}(k)$ may have either sign.

If $M\ge n$, every divisor occurs and $\F_{n,M}=\ell_n^{(n)}$, so the positivity is immediate. We henceforth assume $M<n$. The second consequence of the Foulkes formula is that two natural divisor prefixes require no estimates at all.

For $M<n$, write
\[
 D_{n,M}:=\{d:d\mid n,\ d\le M\}.
\]
The following two divisor prefixes are already genuine characters.

\begin{proposition}
Let $M<n$.
\begin{enumerate}
\item If $D_{n,M}$ is the complete divisor set of some $K\mid n$, then $\F_{n,M}=\ell_n^{(K)}$.
\item If $D_{n,M}$ contains all proper divisors of $n$, then $\F_{n,M}=\ell_n^{(n)}-p_n$ is Schur-positive.
\end{enumerate}
\label{prop:prefixes}
\end{proposition}

\begin{proof}
\smallskip\noindent\emph{Complete divisor sets.}
If $D_{n,M}=\{d:d\mid K\}$, then \eqref{eq:Foulkes} with $k=K$ gives
\[
 \F_{n,M}=\sum_{d\mid K}\Lie_{n/d}[p_d]=\ell_n^{(K)}.
\]
\smallskip\noindent\emph{All proper divisors.}
For the second assertion,
\[
 \F_{n,M}=\sum_{\substack{d\mid n\\d<n}}\Lie_{n/d}[p_d]
 =\ell_n^{(n)}-p_n.
\]
For a non-hook, positivity is immediate because $\ell_n^{(n)}$ is a genuine character and $[s_\lambda]p_n=0$.  To justify the latter assertion, the Murnaghan--Nakayama rule computes $[s_\lambda]p_n=\chi^\lambda(n)$ by removing a single rim hook of size $n$ from $\lambda$. Such a removal is possible exactly when the whole Young diagram is a hook, say $\lambda=(n-r,1^r)$; its height is $r+1$, so the rule gives the sign $(-1)^r$. So $[s_\lambda]p_n$ vanishes for a nonhook and equals $(-1)^r$ for this hook.

It remains to consider hooks, where $p_n$ has coefficient $\pm1$. We spell out the induced-character calculation because it also explains why the required correction is present. Since $\chi_n^{(n)}$ is the trivial character of $C_n$, one has
\[
 \ell_n^{(n)}=\ch\operatorname{Ind}_{C_n}^{S_n}\1_{C_n}.
\]
Consequently, equations~\eqref{eq:Schur-multiplicity} and~\eqref{eq:Frobenius-coefficient} give
\begin{align*}
 [s_{(n-r,1^r)}]\ell_n^{(n)}
 &=\left\langle
   \operatorname{Ind}_{C_n}^{S_n}\1_{C_n},
   \chi^{(n-r,1^r)}\right\rangle_{S_n}\\
 &=\left\langle
   \1_{C_n},
   \operatorname{Res}_{C_n}^{S_n}\chi^{(n-r,1^r)}
   \right\rangle_{C_n}\\
 &=\dim\bigl(S^{(n-r,1^r)}\bigr)^{C_n}.
\end{align*}
The last equality holds because the inner product with the trivial character is the dimension of the fixed subspace. Since $C_n$ is generated by $\sigma_n$, this fixed subspace is precisely the eigenspace on which $\sigma_n$ has eigenvalue $1$.

A standard realization of hook Specht modules identifies
\[
 S^{(n-r,1^r)}\cong\bigwedge^r V,
\]
where $V$ is the standard $(n-1)$-dimensional representation; see, for example, \cite{Sagan}. To see the eigenvalues explicitly, let $\mathbb C^n$ be the permutation representation with basis $e_1,\ldots,e_n$. A Fourier basis diagonalizes the cyclic shift $\sigma_n$: its eigenvalues on $\mathbb C^n$ are $1,\zeta_n,\ldots,\zeta_n^{n-1}$. The $1$-eigenline is spanned by $e_1+\cdots+e_n$, and removing this trivial summand leaves the standard representation $V$. Hence the eigenvalues of $\sigma_n$ on $V$ are
\[
 \zeta_n,\zeta_n^2,\ldots,\zeta_n^{n-1}.
\]
If $v_{a_1},\ldots,v_{a_r}$ are corresponding eigenvectors, then
\[
 \sigma_n(v_{a_1}\wedge\cdots\wedge v_{a_r})
 =\zeta_n^{a_1+\cdots+a_r}
  (v_{a_1}\wedge\cdots\wedge v_{a_r}).
\]
Thus a choice of distinct exponents whose sum is divisible by $n$ produces a $C_n$-fixed vector in $\bigwedge^rV$.

We now compare this fixed-space dimension with the coefficient of $p_n$. If $r$ is odd, subtracting $[s_{(n-r,1^r)}]p_n=-1$ adds $1$ to the nonnegative fixed-space dimension. If $r=2s$ is even, the empty wedge handles $s=0$; for $s>0$, choose the exponents
\[
 1,n-1,2,n-2,\ldots,s,n-s.
\]
They are distinct because $2s=r\le n-1$ implies $s<n/2$, and their sum is $sn$, a multiple of $n$. The associated wedge is therefore fixed by $\sigma_n$. Hence $[s_{(n-r,1^r)}]\ell_n^{(n)}\ge1$, which absorbs the coefficient $+1$ contributed by $p_n$ when $r$ is even.
\end{proof}

Call a cutoff \emph{proper} if $M<n$, and call a proper cutoff \emph{elementary} if it is one of the two types in Proposition~\ref{prop:prefixes}. From this point onward, only non-elementary proper cutoffs require analysis. The next lemma identifies their first possible degrees. These values will later appear as natural left endpoints of the estimates, rather than as thresholds chosen to make the inequalities work. Let $\wpr(n)$ denote the number of distinct prime divisors of $n$.

\begin{lemma}[First non-elementary degrees]
If a proper cutoff is not elementary, then $n\ge12$. More precisely:
\begin{enumerate}
\item if $\wpr(n)\le2$, then either $n=12$ or $n\ge18$;
\item if $\wpr(n)\ge3$, then $n\ge30$.
\end{enumerate}
\label{lem:arithmetic-left-endpoints}
\end{lemma}

\begin{proof}
For a prime power, every divisor prefix is the complete divisor set of a divisor. If $n=pq$ with distinct primes $p<q$, the only proper prefixes are $\{1\}$, $\{1,p\}$, and the set of all proper divisors; they are elementary. Thus a non-elementary cutoff with at most two distinct prime divisors must have the form $p^a q^b$ with $a,b\ge1$ and $a+b\ge3$. The least possibility is $2^2\cdot3=12$. The next is $2\cdot3^2=18$: using a prime at least $5$ gives at least $2^2\cdot5=20$, while increasing the exponent of $2$ in the pair $\{2,3\}$ gives at least $2^3\cdot3=24$. If at least three distinct primes divide $n$, then $n\ge2\cdot3\cdot5=30$.
\end{proof}

\section{Arithmetic and rectangular-cycle estimates}
\label{sec:estimates}

The nonidentity terms in \eqref{eq:central-character} involve only rectangular cycle types. Two independent inputs will control them: a character estimate for these rectangular classes, and an arithmetic estimate for the weights $\psi_M(d)$. We record the two inputs separately, beginning with the character side.

Fomin and Lulov proved the following estimate.

\begin{theorem}[Fomin--Lulov]
Let $\lambda\vdash n$ and write $n=ds$. Then
\begin{equation}
 |\chi^\lambda(d^s)|
 \le \frac{s!d^s}{(n!)^{1/d}}(f^\lambda)^{1/d}.
 \label{eq:FL}
\end{equation}
\end{theorem}

This is \cite[Theorem 1.1]{FominLulov}; see also \cite[Theorem 1.7]{Swanson}. We use the following normalized form.

\begin{lemma}[Normalized Fomin--Lulov bound]
Let $\lambda\vdash n=ds$. Then
\begin{equation}
 \frac{|\chi^\lambda(d^s)|}{f^\lambda}
 \le d^{-1/2}e^{d/(12n)}
 \left(\frac{\sqrt{2\pi n}}{f^\lambda}\right)^{1-1/d}.
 \label{eq:normalized-FL}
\end{equation}
\label{lem:normalized-FL}
\end{lemma}

\begin{proof}
If $\chi^\lambda(d^s)=0$, there is nothing to prove. Otherwise, use the two-sided Stirling bounds
\begin{equation}
 \left(m+\frac12\right)\log m-m+\log\sqrt{2\pi}
 \le\log(m!)
 \le\left(m+\frac12\right)\log m-m+\log\sqrt{2\pi}+\frac1{12m},
 \label{eq:Stirling}
\end{equation}
which follow from Robbins's explicit remainder estimate \cite{Robbins}. Divide \eqref{eq:FL} by $f^\lambda$, put $s=n/d$, apply the upper bound to $s!$ and the lower bound to $n!$, and take logarithms. The terms linear in $n$ cancel, leaving
\[
 \log\frac{|\chi^\lambda(d^s)|}{f^\lambda}
 \le\left(1-\frac1d\right)
 \left(\frac12\log n-\log f^\lambda+\log\sqrt{2\pi}\right)
 +\frac{d}{12n}-\frac12\log d.
\]
Exponentiation gives \eqref{eq:normalized-FL}.
\end{proof}

The important feature of \eqref{eq:normalized-FL} is the negative power of $f^\lambda$: once the dimension is large, every nonidentity rectangular character value is small relative to the identity value $f^\lambda$.

The global arithmetic bound comes from the following finite-divisor form of Landau's inequality, proved by Tao.

\begin{lemma}[Tao's finite-divisor Möbius bound]
For every integer $N\ge1$ and every real $x\ge1$,
\begin{equation}
 \left|\sum_{\substack{q\mid N\\q\le x}}\frac{\mu(q)}q\right|\le1.
 \label{eq:Tao-Mobius}
\end{equation}
The same bound holds with $q<x$ for every real $x>0$.
\label{lem:Tao-Mobius}
\end{lemma}

The displayed finite-divisor form is
\cite[equation~(10)]{TaoMobius}, stated there as a corollary
of \cite[Theorem~1.1]{TaoMobius}. For the strict-cutoff version, the sum is empty when $0<x\le1$.  When $x>1$, choose $y<x$ so close to $x$ that no divisor of $N$ lies in $[y,x)$; then $q<x$ is equivalent to $q\le y$, and \eqref{eq:Tao-Mobius} applies to $y\ge1$.

\begin{lemma}[Bounds for the truncated M\"obius weights]
For every $d,M\ge1$,
\begin{equation}
 |\psi_M(d)|\le2d.
 \label{eq:psi-global}
\end{equation}
If $d$ has at most two distinct prime divisors, then the sharper bound
\begin{equation}
 |\psi_M(d)|\le
 \max\{\varphi(d),d-\varphi(d)\}\le d
 \label{eq:psi-two-prime}
\end{equation}
holds. In particular, $|\psi_M(2)|\le1$ for every $M$.
\label{lem:psi-bounds}
\end{lemma}

\begin{proof}
If $d\le M$, then all divisors of $d$ occur in the defining sum, and the convolution identity $\varphi=\operatorname{id}*\mu$ gives $\psi_M(d)=\varphi(d)$, where $\varphi$ is Euler's totient function. Suppose $d>M$ and put $x=d/M>1$. The change of variables $q=d/m$ gives
\[
 \frac{\psi_M(d)}d
 =\sum_{\substack{q\mid d\\q\ge x}}\frac{\mu(q)}q
 =\frac{\varphi(d)}d-
   \sum_{\substack{q\mid d\\q<x}}\frac{\mu(q)}q.
\]
The first term lies in $[0,1]$, and the second has absolute value at most $1$ by Lemma~\ref{lem:Tao-Mobius}. This proves \eqref{eq:psi-global}.

For the sharper assertion, the case $d=1$ is immediate, so assume $d>1$.  The cutoff $m\le M$ selects an initial segment of the nonzero summands when they are ordered by $m$. If $d=p^a$, the only nonzero summands occur at $m=p^{a-1}$ and $m=p^a$; the possible partial sums are
\[
 0,\quad -p^{a-1},\quad \varphi(d).
\]
Now let $d=p^a q^b$ with $p<q$, and put $A=p^{a-1}q^{b-1}$. The only nonzero summands occur, in increasing order, at
\[
 A,\quad pA,\quad qA,\quad pqA=d.
\]
The successive partial sums are
\[
 0,\quad A,\quad -(p-1)A,\quad -(p+q-1)A,\quad (p-1)(q-1)A.
\]
The last two absolute values are $d-\varphi(d)$ and $\varphi(d)$, and they dominate the preceding ones. This proves \eqref{eq:psi-two-prime}.
\end{proof}

We also use the following convenient divisor estimate.

\begin{lemma}[A divisor-sum estimate]
For every integer $n\ge1$,
\begin{equation}
 \sum_{d\mid n}\sqrt d\le3n^{3/4}.
 \label{eq:divisor-sum}
\end{equation}
\label{lem:divisor-sum}
\end{lemma}

\begin{proof}
Put $N=\lfloor\sqrt n\rfloor$. Pair each divisor $d\le\sqrt n$ with $n/d$ and enlarge the set of lower divisors to all $1\le j\le N$. Then
\[
 \sum_{d\mid n}\sqrt d
 \le\sum_{j=1}^{N}\left(\sqrt j+\sqrt{\frac nj}\right)
 \le N\sqrt N+\sqrt n\left(1+\int_1^N x^{-1/2}\,dx\right)
 \le3n^{3/4}.
\]
Counting $\sqrt n$ twice when $n$ is a square only enlarges the right-hand side.
\end{proof}

The majority argument will use the arithmetic bounds together with an exact localization to short cycles. The balanced argument will instead combine all three estimates with a global lower bound for $f^\lambda$.

\section{Majority diagrams}
\label{sec:majority}

Assume first that the first row contains more than half of the boxes. Write
\[
 \lambda=(n-r,\mu),
 \qquad |\mu|=r<\frac n2.
\]
The point of the majority condition is not that $r$ is bounded, but that the first row can be created in one step by a Bernstein operator. This simultaneously controls every distance from the row boundary.

The argument has three steps. We first show that cycles longer than the tail size $r$ contribute only a simple hook constant. We then compare $f^{(n-r,\mu)}$ with a two-row dimension and prove a uniform contraction for the remaining short cycles. Finally, we evaluate the row and column constants and conclude the positivity.

\subsection{Localization to short cycles}

Let $e_j^\perp$ denote the operator adjoint to multiplication by $e_j$ in the Hall inner product, and put
\[
 \delta_\mu=
 \begin{cases}
 1,&\mu=(1^r),\\
 0,&\text{otherwise}.
 \end{cases}
\]

\begin{lemma}[ A Bernstein localization]
Let $d\mid n$. There is a decomposition
\begin{equation}
 \chi^{(n-r,\mu)}(d^{n/d})
 =\delta_\mu(-1)^r+\eta_d
 \label{eq:Bernstein-decomposition}
\end{equation}
with
\begin{equation}
 \eta_d=0\quad(d>r),
 \qquad
 |\eta_d|\le f^\mu
 \sum_{q=1}^{\lfloor r/d\rfloor}\binom{n/d}{q}.
 \label{eq:eta-bound}
\end{equation}
\label{lem:Bernstein-character}
\end{lemma}

\begin{proof}
Because $n-r>r\ge\mu_1$, the Bernstein operator that adjoins a first part $n-r$ produces the genuine partition $(n-r,\mu)$.  The Bernstein creation formula \cite[Chapter I, \S5]{Macdonald} therefore gives
\begin{equation}
 s_{(n-r,\mu)}
 =\sum_{j=0}^{r}(-1)^j h_{n-r+j}e_j^\perp s_\mu.
 \label{eq:Bernstein-creation}
\end{equation}
By the adjoint Pieri rule,
\[
 e_j^\perp s_\mu
 =\sum_{\substack{\nu\subseteq\mu\\
        \mu/\nu\text{ is a vertical }j\text{-strip}}}s_\nu,
\]
where a vertical strip has at most one box in each row.
\smallskip\noindent\emph{The induced-character value.}
The product $h_{n-r+j}s_\nu$ is the Frobenius characteristic of
\[
 W_\nu:=\operatorname{Ind}_{S_{n-r+j}\times S_{r-j}}^{S_n}
       (\1\otimes S^\nu).
\]
It is useful to make the value of this induced character explicit. One may realize $W_\nu$ as a direct sum of copies of $S^\nu$ indexed by the $(r-j)$-element subsets $A\subseteq\{1,\ldots,n\}$: a permutation sends the summand indexed by $A$ to the summand indexed by its image. Thus only subsets fixed setwise by the permutation contribute to the trace. If the permutation has type $d^{n/d}$, a fixed subset must be a union of whole $d$-cycles. So the character value is zero unless $d\mid r-j$. When $r-j=dq$, there are $\binom{n/d}{q}$ choices of $q$ cycles; on each corresponding summand the restricted permutation has cycle type $d^q$ and trace $\chi^\nu(d^q)$, while the complementary factor is trivial. Thus the value is
\[
 \binom{n/d}{q}\chi^\nu(d^q).
\]
Substitution into \eqref{eq:Bernstein-creation} yields
\begin{equation}
 \chi^{(n-r,\mu)}(d^{n/d})
 =\sum_{q=0}^{\lfloor r/d\rfloor}
 (-1)^{r-dq}\binom{n/d}{q}
 \sum_{\substack{\nu\subseteq\mu\\
 \mu/\nu\text{ vertical of size }r-dq}}
 \chi^\nu(d^q).
 \label{eq:Bernstein-character-full}
\end{equation}
\smallskip\noindent\emph{The constant term.}
The term $q=0$ is present exactly when $\mu$ itself is a vertical strip, and then equals $(-1)^r$. This is the constant in \eqref{eq:Bernstein-decomposition}; all remaining terms vanish when $d>r$.

\smallskip\noindent\emph{Bounding the remainder.}
For $s\le r$, the iterated branching rule says that the multiplicity of $S^\nu$ in the restriction of $S^\mu$ from $S_r$ to $S_s$ is the number $f^{\mu/\nu}$ of standard tableaux of skew shape $\mu/\nu$. Taking dimensions gives
\[
 f^\mu=\sum_{\substack{\nu\subseteq\mu\\|\nu|=s}}
 f^\nu f^{\mu/\nu}.
\]
Every skew Young diagram has at least one standard tableau: take any linear extension of the usual northwest-to-southeast partial order on its boxes.  Hence $f^{\mu/\nu}\ge1$ for every occurring subpartition $\nu$.  It follows that the sum of $f^\nu$ over any subfamily of these subpartitions, in particular over the vertical strips in \eqref{eq:Bernstein-character-full}, is at most $f^\mu$. Finally, $|\chi^\nu(\rho)|\le f^\nu$ because a finite-group representation may be chosen unitary and the absolute value of a trace is at most the dimension. Summing the terms with $q\ge1$ proves \eqref{eq:eta-bound}.
\end{proof}

\subsection{Dimension comparison and short-cycle contraction}

For $0\le r<n/2$, set
\begin{equation}
 g_r(n):=\binom nr-\binom n{r-1}
 =\binom nr\frac{n-2r+1}{n-r+1},
 \label{eq:gr}
\end{equation}
where $\binom n{-1}=0$. The hook-length formula for the two-row partition $(n-r,r)$ gives
\[
 f^{(n-r,r)}=\binom nr-\binom n{r-1}=g_r(n),
\]
so the comparison below is naturally normalized by an actual irreducible dimension.

\begin{lemma}[A comparison of dimension]
If $\lambda=(n-r,\mu)$ with $r<n/2$, then
\begin{equation}
 f^\lambda\ge f^\mu g_r(n).
 \label{eq:majority-dimension}
\end{equation}
\label{lem:majority-dimension}
\end{lemma}

\begin{proof}
Put $a=n-r$. Separating the hooks in the first row from those in $\mu$, the hook-length formula gives
\begin{equation}
 f^{(a,\mu)}
 =\binom nr f^\mu
 \prod_{j=1}^{r}\frac{a-j+1}{a-j+1+\mu_j'},
 \label{eq:first-row-hook-factor}
\end{equation}
where $\mu_j'=0$ beyond the last column. The denominator product is largest when $\mu'=(1^r)$. To see this while remaining inside the set of partitions, suppose $\mu'\ne(1^r)$, choose $i$ to be the last index with $\mu_i'\ge2$, and let $j$ be the first index with $\mu_j'=0$. Replacing these entries by $\mu_i'-1$ and $1$ preserves weak decrease and the total sum.  If $D$ and $D_{\mathrm{new}}$ denote the old and new denominator products in \eqref{eq:first-row-hook-factor}, respectively, then all but two factors cancel.  Consequently, $D_{\mathrm{new}}-D$ has the same sign as
\[
 \begin{aligned}
 &(a-i+\mu_i')(a-j+2)\\
 &\qquad -(a-i+1+\mu_i')(a-j+1)
   =j-i+\mu_i'-1>0.
 \end{aligned}
\]
Thus each transfer increases the denominator product; iteration reaches $(1^r)$.  Since the numerator in \eqref{eq:first-row-hook-factor} is fixed, the product of ratios there is therefore at least
\[
 \prod_{j=1}^{r}\frac{a-j+1}{a-j+2}
 =\frac{a-r+1}{a+1}
 =\frac{n-2r+1}{n-r+1}.
\]
Together with \eqref{eq:gr}, this proves the claim.
\end{proof}

The next proposition is the short-cycle contraction behind all majority shapes. Its degree endpoints are supplied by Lemma~\ref{lem:arithmetic-left-endpoints}. The proof has three natural ranges. At $n=12$ the tail has size at most $5$; Initially, the deciding threshold is simply the arithmetic lower bound of 18. But once $r\ge 9$, it's the majority condition $n\ge 2r+1$ that becomes the critical cutoff.

\begin{lemma}[Monotonicity of the short-cycle ratio]
Fix $r\ge2$, $2\le d\le r$, and $1\le q\le\lfloor r/d\rfloor$. Then
\begin{equation}
 \frac{\binom{\lfloor n/d\rfloor}{q}}{g_r(n)}
 \label{eq:binomial-ratio}
\end{equation}
is nonincreasing in the integers $n\ge2r+1$.
\label{lem:short-cycle-monotonicity}
\end{lemma}

\begin{proof}
From \eqref{eq:gr},
\[
 \frac{g_r(n+1)}{g_r(n)}
 =\frac{(n+1)(n-2r+2)}{(n-r+2)(n-2r+1)}>1,
\]
since the numerator minus the denominator after clearing factors is $r(n-2r+3)>0$. Thus there is nothing to prove when $\lfloor(n+1)/d\rfloor=\lfloor n/d\rfloor$.

At a jump, write $n+1=d(N+1)$, so that the binomial numerator changes from $\binom Nq$ to $\binom{N+1}q$.  If $R_n$ denotes \eqref{eq:binomial-ratio}, then
\begin{align*}
 \frac{R_{n+1}}{R_n}
 &=\frac{(n-r+2)(n-2r+1)}{(n+1)(n-2r+2)}
   \frac{N+1}{N+1-q}\\
 &=\frac{(n-r+2)(n-2r+1)}{(n-r+1)(n-2r+2)}
   \left(\frac{N+1}{N+1-q}\frac{n-r+1}{n+1}\right).
\end{align*}
The first factor is smaller than $1$ (the denominator exceeds the numerator by $r$ after expansion).  The parenthesized factor is at most $1$ precisely when
\[
 q(n+1)\le r(N+1),
\]
and this is equivalent to $qd\le r$ because $n+1=d(N+1)$.  The latter inequality holds by the choice of $q$.
\end{proof}

The actual character sum contains only divisors of $n$. In the next proposition we enlarge it to all $2\le d\le r$; the floor in the binomial coefficient makes this universal majorant meaningful even when $d\nmid n$. The factor $2/3$ is chosen for a structural reason: in the admissible range it leaves one third of the dimension, enough to absorb the remaining trivial or sign-twisted hook constant.

\begin{proposition}[Short-cycle contraction]
Let $M<n$, and suppose that $D_{n,M}$ is not elementary. For every integer $r$ with $2\le r<n/2$,
\begin{equation}
 \sum_{d=2}^{r}|\psi_M(d)|
 \sum_{q=1}^{\lfloor r/d\rfloor}
 \binom{\lfloor n/d\rfloor}{q}
 <\frac23 g_r(n).
 \label{eq:short-cycle-contraction}
\end{equation}
\label{prop:short-cycle-contraction}
\end{proposition}

\begin{proof}
For $c>0$, write
\[
 A_{n,r}(c):=\sum_{d=2}^{r}cd
 \sum_{q=1}^{\lfloor r/d\rfloor}
 \binom{\lfloor n/d\rfloor}{q}.
\]
\smallskip\noindent\emph{The first admissible degree $n=12$.}
If $n=12$, then $r\le5$, and every integer $d\le r$ has at most two distinct prime factors. Lemma~\ref{lem:psi-bounds} therefore bounds the left-hand side of \eqref{eq:short-cycle-contraction} by $A_{12,r}(1)$. For $r=2$, this majorant is $12$, while $g_2(12)=54$. For $3\le r\le5$, it is at most $A_{12,5}(1)$, namely
\[
 2\left(\binom61+\binom62\right)
 +3\binom41+4\binom31+5\binom21=76.
\]
Formula
\begin{equation}
 \frac{g_{r+1}(n)}{g_r(n)}
 =\frac{(n-r+1)(n-2r-1)}{(r+1)(n-2r+1)}
 \label{eq:gr-ratio}
\end{equation}
shows that $g_r(12)$ increases for $3\le r\le5$, so $g_r(12)\ge g_3(12)=154$. The inequalities $3\cdot12<2\cdot54$ and $3\cdot76<2\cdot154$ prove the required factor $2/3$.

\smallskip\noindent\emph{The fixed left endpoint $n=18$.}
It remains to consider $n\ge18$. The uniform bound $|\psi_M(d)|\le2d$ replaces the left-hand side by $A_{n,r}(2)$. For fixed $r$, Lemma~\ref{lem:short-cycle-monotonicity} shows that every summand of $A_{n,r}(2)/g_r(n)$ is nonincreasing in $n$, and hence so is their sum. The admissible numerical range also requires $n\ge2r+1$. Therefore it is enough to evaluate the quotient at its left endpoint
\[
 n_0=\max\{18,2r+1\}.
\]
The integer $n_0$ need not itself support a non-elementary cutoff; this is immaterial, because the quotient is a universal majorant and is nonincreasing throughout the whole numerical range.
For $r=2,3$, $A_{18,r}(2)$ is $36$ and $72$, respectively, and $g_2(18)=135$, $g_3(18)=663$. For $4\le r\le8$, the majorant is increasing with $r$ and is at most $A_{18,8}(2)=1352$. Now $g_4(18)=2244$.  For $5\le r\le8$, one has $\binom{18}{r}\ge\binom{18}{5}$ and $(19-2r)/(19-r)\ge3/11$, so
\[
 g_r(18)=\binom{18}{r}\frac{19-2r}{19-r}
 \ge\binom{18}{5}\frac3{11}>2244;
\]
the last inequality is the integer comparison $3\cdot8568>11\cdot2244$.  Thus $3\cdot1352<2\cdot2244$, and the contraction holds throughout this range.

\smallskip\noindent\emph{The moving left endpoint $n=2r+1$.}
Suppose finally that $r\ge9$, so $n_0=2r+1$. Put $C_r=\binom{2r+1}{r}$. Since every binomial partial sum is at most the full binomial sum,
\begin{align}
 \sum_{d=2}^{r}2d
 \sum_{q=1}^{\lfloor r/d\rfloor}
 \binom{\lfloor(2r+1)/d\rfloor}{q}
 &\le4\,2^r
 +\sum_{d=3}^{r}2d\,2^{(2r+1)/3}\notag\\
 &<4\,2^r+r(r+1)2^{(2r+1)/3}.
 \label{eq:majority-tail}
\end{align}
Moreover,
\begin{equation}
 C_r\ge\frac{4^r}{\sqrt r}\qquad(r\ge5).
 \label{eq:central-binomial-lower}
\end{equation}
Indeed, if $a_r=C_r\sqrt r/4^r$, then
\[
 \left(\frac{a_{r+1}}{a_r}\right)^2
 =\frac{(2r+3)^2(r+1)}{4r(r+2)^2}
 =1+\frac{5r+9}{4r(r+2)^2}>1.
\]
Thus $a_r$ is increasing.  The initial case $a_5>1$ follows on squaring from
$5\cdot462^2>1024^2$.

Since $g_r(2r+1)=2C_r/(r+2)$, equations \eqref{eq:majority-tail} and \eqref{eq:central-binomial-lower} bound the quotient in \eqref{eq:short-cycle-contraction} by
\begin{equation}
 V_r:=\frac{(r+2)\sqrt r}{2}
 \left(4\,2^{-r}+r(r+1)2^{(1-4r)/3}\right).
 \label{eq:Vr}
\end{equation}
Both summands in $V_r$ decrease for $r\ge9$. Their consecutive ratios are, respectively,
\[
 \frac{r+3}{2(r+2)}\sqrt{\frac{r+1}{r}}
 \quad\text{and}\quad
 2^{-4/3}\frac{r+3}{r}\sqrt{\frac{r+1}{r}}.
\]
For $r\ge9$, the first is at most $(6/11)(4/3)=8/11<1$.  For the second, $2^{4/3}>5/2$, $(r+3)/r\le4/3$, and $\sqrt{(r+1)/r}<4/3$, so it is smaller than $32/45$.  Finally,
\[
 90\,2^{-35/3}
 =\frac{90}{2^{11}2^{2/3}}
 <\frac{90}{2^{11}(3/2)}=\frac{15}{512},
\]
where $2^{2/3}>3/2$.  And we have
\[
 V_9
 <\frac{33}{2}\left(\frac4{512}+\frac{15}{512}\right)
 =\frac{627}{1024}<\frac23.
\]
So $V_r<2/3$ for every $r\ge9$.
\end{proof}

\subsection{Hook constants and positivity}

The short-cycle estimate controls the remainder. What remains is to identify the constant term in the row and column cases. The two identities below do exactly that.

\begin{lemma}[The row and column constants]
If $M<n$, then
\begin{equation}
 \sum_{d\mid n}\psi_M(d)=0.
 \label{eq:psi-cancellation}
\end{equation}
Moreover, if
\[
 \eps_{n,M}:=[s_{(1^n)}]\F_{n,M},
\]
then
\begin{equation}
 \eps_{n,M}=
 \begin{cases}
 1,&n\text{ is even and }n/2\le M<n,\\
 0,&\text{otherwise}.
 \end{cases}
 \label{eq:sign-endpoint}
\end{equation}
\label{lem:hook-constants}
\end{lemma}

\begin{proof}
Interchanging the two divisor sums gives
\[
 \sum_{d\mid n}\psi_M(d)
 =\sum_{\substack{m\mid n\\m\le M}}
 m\sum_{q\mid n/m}\mu(q)=0,
\]
because the M\"obius inversion identity used in Lemma~\ref{lem:theta} makes the inner sum equal to $1$ when $n/m=1$ and to $0$ otherwise. Thus only $m=n$ could contribute, and that value is excluded by $M<n$.

We next compute the sign coefficient without suppressing the reciprocity step. Since
\[
 \ell_n^{(k)}=\ch\operatorname{Ind}_{C_n}^{S_n}\chi_n^{(k)},
\]
equations~\eqref{eq:Schur-multiplicity} and~\eqref{eq:Frobenius-coefficient} give
\begin{align*}
 [s_{(1^n)}]\ell_n^{(k)}
 &=\left\langle
   \operatorname{Ind}_{C_n}^{S_n}\chi_n^{(k)},
   \operatorname{sgn}\right\rangle_{S_n}\\
 &=\left\langle
   \chi_n^{(k)},
   \operatorname{Res}_{C_n}^{S_n}\operatorname{sgn}
   \right\rangle_{C_n}.
\end{align*}
Both characters in the last inner product are one-dimensional, so the value is $1$ exactly when they agree and is $0$ otherwise. Now
\[
 \operatorname{sgn}(\sigma_n)=(-1)^{n-1}.
\]
If $n$ is odd, this value is $1=\zeta_n^n$, so the restricted sign character is $\chi_n^{(n)}$; if $n$ is even, it is $-1=\zeta_n^{n/2}$, so the restriction is $\chi_n^{(n/2)}$. Substituting the decomposition \eqref{eq:theta-decomposition}, only the matching Foulkes character can contribute. Thus, for odd $n$,
\[
 \eps_{n,M}=\theta_{M,n}(n)=0,
\]
because $M<n$ excludes the only possible term $j=1$ in the definition of $\theta_{M,n}(n)$. For even $n$,
\[
 \eps_{n,M}=\theta_{M,n}(n/2)
 =\sum_{\substack{j\mid2\\(n/2)j\le M}}\mu(j).
\]
The term $j=1$ is present exactly when $M\ge n/2$, whereas the term $j=2$ is absent because $M<n$. Hence this sum equals $1$ precisely when $n/2\le M<n$ and equals $0$ otherwise.
\end{proof}

\begin{proposition}[Majority positivity]
Let $M<n$, suppose that $D_{n,M}$ is not elementary, and let $\lambda\vdash n$. If
\[
 \max(\lambda_1,\lambda_1')>\frac n2,
\]
then $[s_\lambda]\F_{n,M}\ge0$. More precisely, if $\lambda=(n-r,\mu)$ or $\lambda'=(n-r,\mu)$ with $2\le r<n/2$, then the coefficient is strictly positive.
\label{prop:majority}
\end{proposition}

\begin{proof}
For $r\ge2$, Lemmas~\ref{lem:Bernstein-character} and~\ref{lem:majority-dimension}, together with Proposition~\ref{prop:short-cycle-contraction}, first give a uniform error estimate; we then insert the appropriate row or column constant.

\smallskip\noindent\emph{The common error bound.}
Suppose $\lambda=(n-r,\mu)$ or $\lambda'=(n-r,\mu)$ with $r<n/2$. By \eqref{eq:eta-bound}, enlarging the divisor sum to all integers $2\le d\le r$ gives
\begin{align*}
 \left|
 \sum_{\substack{d\mid n\\2\le d\le r}}
 \psi_M(d)\eta_d\right|
 &\le f^\mu
 \sum_{d=2}^{r}|\psi_M(d)|
 \sum_{q=1}^{\lfloor r/d\rfloor}
 \binom{\lfloor n/d\rfloor}{q}\\
 &<\frac23 f^\mu g_r(n)
 \le\frac23 f^\lambda,
\end{align*}
where the strict inequality is Proposition~\ref{prop:short-cycle-contraction} and the last inequality is Lemma~\ref{lem:majority-dimension}. Thus
\begin{equation}
 \left|
 \sum_{\substack{d\mid n\\2\le d\le r}}
 \psi_M(d)\eta_d\right|<\frac23 f^\lambda.
 \label{eq:majority-error}
\end{equation}
The same estimate remains valid after multiplying each summand by the sign $(-1)^{n-n/d}$, since these factors have absolute value $1$.

\smallskip\noindent\emph{A majority row.}
We now insert the decomposition term by term into \eqref{eq:central-character}. In the row case, the constant part contributes
\[
 \delta_\mu(-1)^r
 \sum_{\substack{d\mid n\\d>1}}\psi_M(d).
\]
Since $\psi_M(1)=1$, equation~\eqref{eq:psi-cancellation} says that the displayed divisor sum is $-1$. Moreover, $\eta_d=0$ for $d>r$. Therefore
\begin{equation}
 n[s_{(n-r,\mu)}]\F_{n,M}
 =f^{(n-r,\mu)}-\delta_\mu(-1)^r
 +\sum_{\substack{d\mid n\\2\le d\le r}}
  \psi_M(d)\eta_d.
 \label{eq:majority-row-coefficient}
\end{equation}

\smallskip\noindent\emph{A majority column.}
Now suppose $\lambda'=(n-r,\mu)$. The standard relation $S^\lambda\cong S^{\lambda'}\otimes\operatorname{sgn}$ gives
\[
 \chi^\lambda(d^{n/d})
 =(-1)^{n-n/d}\chi^{\lambda'}(d^{n/d}),
\]
because a permutation with $n/d$ cycles has sign $(-1)^{n-n/d}$. To evaluate the constant part, pair the power-sum formula \eqref{eq:power-formula} with the sign character. By the definition of $\eps_{n,M}$,
\begin{equation}
 n\eps_{n,M}
 =\sum_{d\mid n}\psi_M(d)(-1)^{n-n/d}.
 \label{eq:signed-psi-sum}
\end{equation}
The $d=1$ term on the right is $\psi_M(1)=1$, so the corresponding sum over $d>1$ equals $n\eps_{n,M}-1$. Inserting the Bernstein decomposition for $\lambda'$ and again using $\eta_d=0$ for $d>r$ gives
\begin{equation}
 n[s_\lambda]\F_{n,M}
 =f^\lambda+\delta_\mu(-1)^r(n\eps_{n,M}-1)
 +\sum_{\substack{d\mid n\\2\le d\le r}}
 (-1)^{n-n/d}\psi_M(d)\eta_d.
 \label{eq:majority-column-coefficient}
\end{equation}
The last sum again has absolute value less than $2f^\lambda/3$.

\smallskip\noindent\emph{Positivity away from the four endpoints.}
If $\delta_\mu=0$, either formula is therefore greater than $f^\lambda/3$. If $\delta_\mu=1$, then the relevant majority shape is a hook.  In the row case it is $(n-r,1^r)$; in the column case its conjugate has this shape, and conjugation preserves dimension.  The exterior-power realization therefore gives in either case
\[
 f^\lambda=f^{(n-r,1^r)}=\dim\bigwedge^rV=\binom{n-1}{r}.
\]
The constant in the row formula has absolute value at most $1$, and that in the column formula has absolute value at most $n-1$. Since a non-elementary cutoff has $n\ge12$ and $2\le r\le(n-1)/2$, the binomial coefficients are increasing from $r=2$ up to the middle: indeed,
\[
 \frac{\binom{n-1}{r+1}}{\binom{n-1}{r}}
 =\frac{n-1-r}{r+1}\ge1
 \qquad\left(r<\frac{n-2}{2}\right).
\]
Consequently $f^\lambda=\binom{n-1}{r}\ge\binom{n-1}{2}$. Hence
\[
 \frac13f^\lambda-(n-1)
 \ge(n-1)\left(\frac{n-2}{6}-1\right)>0.
\]
Thus every majority coefficient with $r\ge2$ is strictly positive.

\smallskip\noindent\emph{The four endpoint shapes.}
When $r=0$ or $1$, the bound $\eta_d=0$ for $d>r$ leaves no remainder at all. For $r=0$, equation~\eqref{eq:majority-row-coefficient} reads $n[s_{(n)}]\F_{n,M}=1-1=0$, while \eqref{eq:majority-column-coefficient} reads $n[s_{(1^n)}]\F_{n,M}=n\eps_{n,M}$. For $r=1$, the hook dimension is $n-1$ and $(-1)^r=-1$, so the same two formulas give
\[
 [s_{(n-1,1)}]\F_{n,M}=1,
 \qquad
 [s_{(2,1^{n-2})}]\F_{n,M}=1-\eps_{n,M}.
\]
Thus the four endpoint coefficients are
\begin{align*}
 [s_{(n)}]\F_{n,M}&=0,
 & [s_{(n-1,1)}]\F_{n,M}&=1,\\
 [s_{(1^n)}]\F_{n,M}&=\eps_{n,M},
 & [s_{(2,1^{n-2})}]\F_{n,M}&=1-\eps_{n,M}.
\end{align*}
They are nonnegative by Lemma~\ref{lem:hook-constants}.
\end{proof}

\section{Balanced diagrams}
\label{sec:balanced}

Assume now that neither a row nor a column has a majority:
\begin{equation}
 \lambda_1\le\frac n2,
 \qquad
 \lambda_1'\le\frac n2.
 \label{eq:balanced-condition}
\end{equation}
The dimension is then automatically of central-binomial size. Unlike the majority case, there is no hook constant to analyze: a balanced diagram cannot itself be a hook. We first establish the dimension bound, and then use it to make the identity term dominate every nonidentity rectangular class simultaneously.

\subsection{A Catalan-scale dimension bound}

Following Swanson, define the opposite hook length of a box $(i,j)$ by $i+j-1$, and put
\[
 N(\lambda)=n-\max_{(i,j)\in\lambda}(i+j-1),
\]
\[
 N^*(\lambda)=
 \begin{cases}
 N(\lambda),&2N(\lambda)+1\le n,\\[2pt]
 \left\lfloor\dfrac{n-1}{2}\right\rfloor,&2N(\lambda)+1>n.
 \end{cases}
\]
Swanson's estimate states that
\begin{equation}
 f^\lambda\ge\frac1{m+1}\binom nm
 \qquad(0\le m\le N^*(\lambda));
 \label{eq:Swanson-opposite}
\end{equation}
see \cite[Corollary 4.13]{Swanson}.

Put
\begin{equation}
 m_n=\left\lfloor\frac{n-1}{2}\right\rfloor,
 \qquad
 B_n=\frac1{m_n+1}\binom n{m_n}.
 \label{eq:Bn}
\end{equation}

\begin{lemma}[Balanced dimension bound]
Every partition satisfying \eqref{eq:balanced-condition} has
\begin{equation}
 f^\lambda\ge B_n.
 \label{eq:balanced-dimension}
\end{equation}
\label{lem:balanced-dimension}
\end{lemma}

\begin{proof}
A box $(i,j)\in\lambda$ forces an $i\times j$ rectangle inside the diagram, so $ij\le n$. Also $i,j\le n/2$ by \eqref{eq:balanced-condition}. If $i,j\ge2$, then
\[
 (i-2)(j-2)\ge0
 \quad\Longrightarrow\quad
 i+j\le\frac{ij}{2}+2\le\frac n2+2.
\]
If $i=1$, then $i+j\le1+\lambda_1\le1+n/2<n/2+2$, and the case $j=1$ is identical using $\lambda_1'$. If $n=2a$, the preceding estimate gives $i+j-1\le a+1=n-m_n$. If $n=2a+1$, integrality gives $i+j\le a+2$, and again $i+j-1\le a+1=n-m_n$. Thus in every case
\[
 i+j-1\le n-m_n.
\]
Hence $N(\lambda)\ge m_n$. If $2N(\lambda)+1\le n$, then $N^*(\lambda)=N(\lambda)\ge m_n$; otherwise $N^*(\lambda)=m_n$. Thus $m=m_n$ is admissible in \eqref{eq:Swanson-opposite}, proving \eqref{eq:balanced-dimension}.
\end{proof}

\begin{lemma}[Growth of the balanced lower bound]
For $n\ge18$,
\begin{equation}
 \frac{B_{n+1}}{B_n}\ge\frac95.
 \label{eq:Bn-growth}
\end{equation}
\label{lem:Bn-growth}
\end{lemma}

\begin{proof}
Writing $n=2a$ or $2a+1$ in \eqref{eq:Bn} gives
\[
 \frac{B_{2a+1}}{B_{2a}}=\frac{2a+1}{a+1},
 \qquad
 \frac{B_{2a+2}}{B_{2a+1}}=\frac{2(a+1)}{a+2}.
\]
For $a\ge9$, both ratios are at least $9/5$.
\end{proof}

\subsection{A global error majorant}

Put
\begin{equation}
 X_n:=\frac{B_n}{\sqrt{2\pi n}},
 \label{eq:balanced-X}
\end{equation}
and, for $\kappa>0$ and $n\ge18$, define
\begin{equation}
 \mathcal G_n(\kappa):=e^{1/12}\left(
 \frac{X_n^{-1/2}}{\sqrt2}
 +3\kappa n^{3/4}X_n^{-2/3}\right).
 \label{eq:Gna}
\end{equation}
The parameter $\kappa$ records whether the available arithmetic estimate is $|\psi_M(d)|\le d$ or $|\psi_M(d)|\le2d$.

\begin{lemma}[Monotonicity of the balanced majorant]
For fixed $\kappa>0$, the sequence $\mathcal G_n(\kappa)$ is strictly decreasing for $n\ge18$.
\label{lem:G-monotonicity}
\end{lemma}

\begin{proof}
Lemma~\ref{lem:Bn-growth} gives
\[
 \frac{X_{n+1}}{X_n}\ge\frac95\sqrt{\frac n{n+1}}>1,
\]
where the last inequality follows after squaring from $81n>25(n+1)$.  Thus the first term decreases. For the second term, Lemma~\ref{lem:Bn-growth} gives
\begin{align*}
 \frac{(n+1)^{3/4}X_{n+1}^{-2/3}}
      {n^{3/4}X_n^{-2/3}}
 &\le
 \left(1+\frac1n\right)^{3/4}
 \left(\frac59\sqrt{1+\frac1n}\right)^{2/3}\\
 &=\left(\frac59\right)^{2/3}
 \left(1+\frac1n\right)^{13/12}<1.
\end{align*}
Indeed, $(5/9)^{2/3}<3/4$ because $(5/9)^2<(3/4)^3$, and, for $n\ge18$,
$(1+1/n)^{13/12}<(19/18)^2<4/3$; the last comparison is simply $361<432$.
\end{proof}

The proof also shows that $X_n$ is increasing for $n\ge18$. At the left endpoint,
\[
 X_{18}=\frac{2431}{3\sqrt\pi}>400,
\]
where we used $\pi<4$. In particular, $X_n>1$ throughout the range in which $\mathcal G_n(\kappa)$ is used.

\begin{proposition}[Balanced positivity]
Let $M<n$, suppose that $D_{n,M}$ is not elementary, and let $\lambda\vdash n$ satisfy \eqref{eq:balanced-condition}. Then
\[
 [s_\lambda]\F_{n,M}>0.
\]
\label{prop:balanced}
\end{proposition}

\begin{proof}
\smallskip\noindent\emph{Removing the longest cycle.}
A hook $\lambda=(a,1^b)$ satisfies $\lambda_1+\lambda_1'=a+(b+1)=n+1$, so one of its first row and first column has length greater than $n/2$. A balanced diagram is therefore not a hook. As explained in the proof of Proposition~\ref{prop:prefixes}, the Murnaghan--Nakayama rule then gives $\chi^\lambda(n)=0$, which removes the term $d=n$ from \eqref{eq:central-character}.

\smallskip\noindent\emph{A common error bound.}
For every remaining divisor $2\le d<n$, Lemma~\ref{lem:normalized-FL} gives
\[
 \frac{|\chi^\lambda(d^{n/d})|}{f^\lambda}
 \le \frac1{\sqrt d}e^{d/(12n)}
 \left(\frac{\sqrt{2\pi n}}{f^\lambda}\right)^{1-1/d}.
\]
Since $f^\lambda\ge B_n$ by Lemma~\ref{lem:balanced-dimension} and $X_n=B_n/\sqrt{2\pi n}$, the last factor is at most $X_n^{-(1-1/d)}$. Multiplying by $|\psi_M(d)|$ and summing gives the intrinsic error bound
\begin{equation}
 \frac1{f^\lambda}
 \sum_{\substack{d\mid n\\d>1}}
 |\psi_M(d)\chi^\lambda(d^{n/d})|
 \le
 \sum_{\substack{d\mid n\\2\le d<n}}
 \frac{|\psi_M(d)|}{\sqrt d}
 e^{d/(12n)}X_n^{-(1-1/d)}.
 \label{eq:balanced-error-master}
\end{equation}
For later use, suppose $n\ge18$ and that $|\psi_M(d)|\le\kappa d$ for every divisor $d\ge3$, while $|\psi_M(2)|\le1$. Since $e^{d/(12n)}\le e^{1/12}$ and $1-1/d\ge2/3$ for $d\ge3$, equation~\eqref{eq:balanced-error-master} and Lemma~\ref{lem:divisor-sum} give
\begin{align}
 \frac1{f^\lambda}
 \sum_{\substack{d\mid n\\d>1}}
 |\psi_M(d)\chi^\lambda(d^{n/d})|
 &\le e^{1/12}\left(
 \frac{X_n^{-1/2}}{\sqrt2}
 +\kappa X_n^{-2/3}
   \sum_{\substack{d\mid n\\d\ge3}}\sqrt d\right)\notag\\
 &\le \mathcal G_n(\kappa).
 \label{eq:G-majorant}
\end{align}
Because $X_n>1$ in this range, replacing the exponent $1-1/d$ by the smaller exponent $2/3$ indeed enlarges $X_n^{-(1-1/d)}$.

\smallskip\noindent\emph{The three arithmetic strata.}
The degrees $12$, $18$, and $30$ below are not thresholds chosen for the estimates: Lemma~\ref{lem:arithmetic-left-endpoints} shows that they are the forced left endpoints of the relevant arithmetic strata. We use only the elementary inequalities $\pi<22/7<4$ and
\begin{equation}
 e^{1/12}<\frac{11}{10},
 \label{eq:exp-elementary}
\end{equation}

\smallskip\noindent\emph{The initial degree $n=12$.}
First suppose $\wpr(n)\le2$. If $n=12$, then $B_{12}=132$ and $X_{12}>15$, since $X_{12}^2=726/\pi>231$. The only nontrivial proper divisors are $2,3,4,6$, and Lemma~\ref{lem:psi-bounds} bounds their weights by $1,2,2,4$. For $d=2$ we retain the exponent $1/2$ in \eqref{eq:balanced-error-master}; for $d=3,4,6$ we use $1-1/d\ge2/3$. Since $X_{12}>15$, this gives
\begin{align}
 \frac1{f^\lambda}
 \sum_{\substack{d\mid12\\d>1}}
 |\psi_M(d)\chi^\lambda(d^{12/d})|
 &<e^{1/12}\left(
 \frac1{\sqrt{30}}
 +\left(\frac2{\sqrt3}+1+\frac4{\sqrt6}\right)15^{-2/3}
 \right)\notag\\
 &<\frac{11}{10}\left(\frac15+\frac23\right)
 =\frac{143}{150}<1.
 \label{eq:n12-contraction}
\end{align}
Indeed, $1/\sqrt{30}<1/5$ and $15^{2/3}>6$, while
\[
 \frac2{\sqrt3}+1+\frac4{\sqrt6}
 <\frac43+1+\frac53=4.
\]

\smallskip\noindent\emph{At most two prime divisors, $n\ge18$.}
Now assume $\wpr(n)\le2$ and $n\ge18$. Every divisor of $n$ has at most two distinct prime factors, so Lemma~\ref{lem:psi-bounds} gives $|\psi_M(d)|\le d$, while $|\psi_M(2)|\le1$. Thus \eqref{eq:G-majorant} applies with $\kappa=1$. At $n=18$,
\[
 B_{18}=4862,
 \qquad
 X_{18}=\frac{2431}{3\sqrt\pi}>400.
\]
Since $18^{3/4}<9$ and $400^{2/3}>54$, the first term in \eqref{eq:Gna} is less than $1/20$, while the second is less than $3\cdot9/54=1/2$. Hence
\begin{equation}
 \mathcal G_{18}(1)
 <\frac{11}{10}\left(\frac1{20}+\frac12\right)
 =\frac{121}{200}<1.
 \label{eq:G18}
\end{equation}
Lemma~\ref{lem:G-monotonicity} proves the claim throughout this arithmetic stratum.

\smallskip\noindent\emph{At least three prime divisors.}
Finally suppose $\wpr(n)\ge3$. Then $n\ge30$, and Lemma~\ref{lem:psi-bounds} gives $|\psi_M(d)|\le2d$. Thus \eqref{eq:G-majorant} applies with $\kappa=2$. Now
\[
 B_{30}=9694845,
 \qquad
 X_{30}=\frac{9694845}{\sqrt{60\pi}}>2^{19};
\]
here $\sqrt{60\pi}<16$ and $9694845>16\cdot2^{19}$. Also $30^{3/4}<13$, and $2^{38/3}=2^{12}2^{2/3}>6144$. The first term in \eqref{eq:Gna} is therefore less than $2^{-10}$, and the second is less than
\[
 6\cdot13\,2^{-38/3}<\frac{78}{6144}=\frac{13}{1024}.
\]
Consequently,
\[
 \mathcal G_{30}(2)
 <\frac{11}{10}\left(2^{-10}+\frac{13}{1024}\right)
 =\frac{77}{5120}<1.
\]
Lemma~\ref{lem:G-monotonicity} handles every larger $n$.

Thus the nonidentity terms in \eqref{eq:central-character} have total absolute value strictly smaller than $f^\lambda$. The identity term therefore dominates, and $[s_\lambda]\F_{n,M}>0$.
\end{proof}

\section{Completion of the proof}
\label{sec:completion}

\begin{proof}[Proof of Theorem~\ref{thm:main}]
If $M\ge n$, then $\F_{n,M}=\ell_n^{(n)}$ is a Foulkes character. Assume $M<n$. If the divisor prefix $D_{n,M}$ is elementary, Proposition~\ref{prop:prefixes} applies. Otherwise Lemma~\ref{lem:arithmetic-left-endpoints} gives $n\ge12$.

Fix $\lambda\vdash n$. If $\max(\lambda_1,\lambda_1')>n/2$, Proposition~\ref{prop:majority} gives a nonnegative coefficient. If $\lambda_1,\lambda_1'\le n/2$, Proposition~\ref{prop:balanced} gives a positive coefficient. These two intrinsic regions cover every partition, so no residual degree or partition remains to be checked. Integrality was established by the Foulkes decomposition \eqref{eq:theta-decomposition}, so all Schur coefficients are nonnegative integers.
\end{proof}

\begin{remark}
At every non-elementary proper cutoff, the proof gives more than nonnegativity. All balanced coefficients are strictly positive, as are all majority coefficients for which at least two boxes lie outside the majority row or column. The remaining four coefficients are
\[
 0,\quad1,\quad\eps_{n,M},\quad1-\eps_{n,M}
\]
for the shapes $(n)$, $(n-1,1)$, $(1^n)$, and $(2,1^{n-2})$, respectively.
\end{remark}

\begin{corollary}
For every $M\ge1$, the element
\[
 \prod_{r=1}^{M}(1-p_r)^{-1}
 =\sum_{\rho:\,\rho_1\le M}p_\rho
\]
is Schur-positive in $\widehat\Lambda_{\mathbb Q}$.
\end{corollary}

\begin{proof}
By Theorem~\ref{thm:main}, each homogeneous input $f_n^{[M]}=\F_{n,M}$ is Schur-positive. Sundaram's completed-series implication \cite[Corollary~3.1]{Sundaram2019b}, applied with $T=[M]$, now shows that $H[\sum_{n\ge1}f_n^{[M]}]$ is Schur-positive. Equation~\eqref{eq:Sundaram-general} identifies this series with the displayed product.
\end{proof}

\section{A separate congruence-class conjecture}
\label{sec:counterexample}

The bounded-interval family studied above is different from a second family considered by Sundaram. Conjecture 6 of \cite{Sundaram2019b}, repeated as Conjecture 5 of \cite{Sundaram2021}, attributes the following statement to Richard Stanley (2015):
\[
 G_k:=\prod_{\substack{m\ge1\\m\equiv1\pmod k}}(1-p_m)^{-1}
\]
is Schur-positive for every $k\ge1$.

Both papers say that it had been checked through degree $24$ for $k\le6$. As printed, however, the statement already fails for $k=3$ in degree $12$. The allowed parts are $1,4,7,10$, and therefore
\[
 (G_3)_{12}
 =p_1^{12}+p_4p_1^8+p_4^2p_1^4+p_4^3
  +p_7p_1^5+p_7p_4p_1+p_{10}p_1^2.
\]
The coefficient of $s_{(1^{12})}$ in $p_\rho$ is the sign $(-1)^{12-\ell(\rho)}$. The seven terms contribute
\[
 1,-1,1,-1,1,-1,-1,
\]
so
\[
 [s_{(1^{12})}](G_3)_{12}=-1.
\]
Thus the displayed product conjecture is false. Since the example lies inside the range reported as checked, this discrepancy suggests that the reported computation used a different formulation or convention, or that the printed statement contains an oversight. In any case, the displayed product has the negative coefficient above.

A nearby interpretation also fails. If $T=\{m:m\equiv1\pmod3\}$ and the intended assertion were that every $f_n^T$ is Schur-positive, then
\[
 f_4^T=\frac14(p_1^4-p_2^2+4p_4),
 \qquad
 [s_{(1^4)}]f_4^T=\frac14(1-1-4)=-1.
\]
These counterexamples do not affect Theorem~\ref{thm:main}: the set $[M]$ is a bounded interval, whereas a congruence class is an unbounded and different set of allowed parts.


\begin{thebibliography}{17}

\bibitem{FominLulov}
S. Fomin and N. Lulov, \emph{On the number of rim hook tableaux}, Zap. Nauchn. Sem. S.-Peterburg. Otdel. Mat. Inst. Steklov. (POMI) 223 (1995), 219--226, 340.

\bibitem{Foulkes}
H. O. Foulkes, \emph{Characters of symmetric groups induced by characters of cyclic subgroups}, Combinatorics, Proc. Conf. Combinatorial Math., Math. Inst., Oxford, 1972, pp. 141--154.

\bibitem{Jollenbeck}
A. J\"ollenbeck, \emph{Cyclic characters of symmetric groups}, J. Algebraic Combin. 12 (2000), 155--161.

\bibitem{Klyachko}
A. A. Klyachko, \emph{Lie elements in the tensor algebra}, Siberian Math. J. 15 (1974), no. 6, 914--920.

\bibitem{Macdonald}
I. G. Macdonald, \emph{Symmetric functions and Hall polynomials}, 2nd ed., Oxford University Press, 1995.

\bibitem{Reutenauer}
C. Reutenauer, \emph{Free Lie algebras}, London Mathematical Society Monographs, New Series, vol. 7, Oxford University Press, 1993.

\bibitem{Robbins}
H. Robbins, \emph{A remark on Stirling's formula}, Amer. Math. Monthly 62 (1955), no. 1, 26--29.

\bibitem{Sagan}
B. E. Sagan, \emph{The symmetric group: representations, combinatorial algorithms, and symmetric functions}, 2nd ed., Graduate Texts in Mathematics, vol. 203, Springer, 2001.

\bibitem{Stanley}
R. P. Stanley, \emph{Enumerative combinatorics, Vol. 2}, Cambridge Studies in Advanced Mathematics, vol. 62, Cambridge University Press, 1999.

\bibitem{Sundaram2018}
S. Sundaram, \emph{The conjugacy action of $S_n$ and modules induced from centralisers}, J. Algebraic Combin. 48 (2018), no. 2, 179--225.

\bibitem{Sundaram2019a}
S. Sundaram, \emph{On a positivity conjecture in the character table of $S_n$}, Electron. J. Combin. 26 (2019), no. 1, Paper No. 1.1, 43 pp.

\bibitem{Sundaram2019b}
S. Sundaram, \emph{On the Schur positivity of sums of power sums}, S\'em. Lothar. Combin. 82B (2019), Art. 49, 12 pp.

\bibitem{Sundaram2020}
S. Sundaram, \emph{On a curious variant of the $S_n$-module $\Lie_n$}, Algebr. Comb. 3 (2020), no. 4, 985--1009.

\bibitem{Sundaram2021}
S. Sundaram, \emph{Prime power variations of higher $\Lie_n$ modules}, J. Combin. Theory Ser. A 184 (2021), 105512.

\bibitem{Swanson}
J. P. Swanson, \emph{On the existence of tableaux with given modular major index}, Algebraic Combinatorics 1 (2018), no. 1, 3--21.

\bibitem{TaoMobius}
T. Tao, \emph{A remark on partial sums involving the M\"obius function}, Bull. Aust. Math. Soc. 81 (2010), no. 2, 343--349.

\bibitem{Thrall}
R. M. Thrall, \emph{On symmetrized Kronecker powers and the structure of the free Lie ring}, Amer. J. Math. 64 (1942), no. 1, 371--388.

\end{thebibliography}
\end{document}